\documentclass[12pt,a4paper]{article}
\usepackage{mathrsfs}
\usepackage{epsf, graphicx}
\usepackage{latexsym,amsfonts,amsbsy,amssymb}
\usepackage{amsmath,amsthm}
\newtheorem{theorem}{Theorem}[section]
\newtheorem{lemma}{Lemma}[section]

\newtheorem{algorithm}{Algorithm}[section]

 \numberwithin{equation}{section}
\allowdisplaybreaks
\begin{document}
\title{ Postprocessing and Higher Order Convergence of Stabilized Finite Element
Discretizations of the Stokes Eigenvalue Problem}
 \author{Hehu Xie\footnote{LSEC, Academy of Mathematics and Systems Science,
 Chinese Academy of Sciences, Beijing 100080, China and Institute
 for Analysis and Computational Mathematics, Otto-von-Guericke-University
 Magdeburg, Postfach 4120, D-39016 Magdeburg, Germany
 (hhxie@lsec.cc.ac.cn).}}
\date{}
\maketitle
\begin{abstract}
In this paper, the stabilized finite element method based on local projection
is applied to discretize the Stokes eigenvalue problems and the corresponding convergence analysis
is given. Furthermore, we also use a method to improve the
convergence rate for the eigenpair approximations of
the Stokes eigenvalue problem. It is based on a postprocessing
strategy that contains solving an additional Stokes source problem
on an augmented finite element space which can be constructed either
by refining the mesh or by using the same mesh but increasing the
order of mixed finite element space. Numerical examples are given to confirm
the theoretical analysis.


\end{abstract}

\section{Introduction}
In this paper, we are concerned with the Stokes eigenvalue problems.
The study of Stokes eigenmodes is required when the dynamics behaviors governed
 by the Navier-Stokes equations result from the way this nonlinear dynamics is controlled by diffusion.
 For the other reasons to study the Stokes eigenmodes,
please read the papers \cite{BatchoKarniadakis, LerricheLabrosse}.

The Stokes eigenvalue problem reads as follows:

Find $({\mathbf u},p, \lambda)$ such that
\begin{equation}\label{stokesproblem}
\left\{
\begin{array}{rcl}
-\Delta{\mathbf u}+\nabla p&=&\lambda{\mathbf u}\ \ \ {\rm in}\ \Omega,\\
\nabla\cdot{\mathbf u}&=&0\ \ \ \ \ {\rm in}\ \Omega,\\
{\mathbf u}&=&0\ \ \ \ \ {\rm on}\ \partial\Omega,\\
\int_{\Omega}{\mathbf u}^2d\Omega&=&1,
 \end{array} \right.
\end{equation}
where $\Omega\subset\mathbb{R}^2$ is a bounded domain with Lipschitz
boundary $\partial\Omega$ and $\Delta$, $\nabla$, $\nabla\cdot$\
denote the Laplacian, gradient and divergence operators,
respectively.

There are several works for the eigenvalue problems and their
numerical methods such as Babu\v{s}ka and Osborn
\cite{BabuskaOsborn1989, BabuskaOsborn, Osborn},
Mercier, Osborn, Rappaz and Raviart
\cite{MercierOsbornRappazRaviart}, etc.
 Osborn \cite{Osborn}, Mercier, Osborn, Rappaz and Raviart
\cite{MercierOsbornRappazRaviart} give an abstract analysis for the
eigenpair approximations by mixed/hybrid finite element methods
based on the general theory of compact operators (\cite{Chatelin}).
 In \cite{HeuvelineRannacher} and \cite{LovadinaLylyStenberg}, a posteriori
error estimates and the corresponding adaptive finite element
methods are given for the Stokes eigenvalue problems.

 The first aim  in this paper is to use the local projection stabilization method to discretize
the Stokes eigenvalue problems.  The local projection stabilization
(LPS) method has been proposed for the Stokes problem in
\cite{BeckerBraack_stokes}. The extension to the transport problem
was given in \cite{BeckerBraack_transport}. The analysis of the
local projection method applied to equal-order interpolation
discretization can be found \cite{MatthiesSkrzypaczTobiska_Oseen}
for Oseen problem and in \cite{MatthiesSkrzypaczTobiska_CDR} for
convection-diffusion problem.\par The stabilization term of the
local projection method is based on a projection $\pi_h: V_h\to
\mathcal{D}_h$ of the finite element space $V_h$ which approximates
the solution into a discontinuous space $\mathcal{D}_h$. The
standard Galerkin discretization is stabilized by adding a term
which gives $L^2$ control over the fluctuation $id-\pi_h$ of the
gradient of the solution. Here, the LPS method is based on the
approximation space $V_h$ and the projection space $\mathcal{D}_h$
are defined on the same mesh. In this case, the approximation space
$V_h$ is enriched to satisfy the local inf-sup condition
guaranteeing the existence of an interpolation with an orthogonal
property compared to standard finite
 element spaces. For more details, please read the book \cite{RoosStynesTobiska}.

Recently, many effective postprocessing methods that improve the
convergence rate for the approximations of the eigenvalue problems
by the finite element methods have been proposed and analyzed
(\cite{AndreevLazarovRacheva, RachevaAndreev, XuZhou}).  Xu and Zhou
\cite{XuZhou} have given a two-grid discretization technique to
improve the convergence rate of the second order elliptic eigenvalue
problems and integral eigenvalue problems. Racheva and Andreev
\cite{RachevaAndreev}, Andreev, Lazarov and Racheva
\cite{AndreevLazarovRacheva} have proposed a postprocessing method
that improve the convergence rate for the numerical approximations
of $2m$-order selfadjoint eigenvalue problems especially biharmonic
eigenvalue problems. In \cite{ChenJiaXIe}, a similar method has been
given for the Stokes eigenvalue problem by mixed finite element
methods. The second aim of this paper is to propose and analyze a
postprocessing algorithm which can improve the convergence rate of
the eigenpair approximations for the Stokes eigenvalue problem by
the LPS method.

The postprocessing procedure can be described as follows: (1)\ solve
the Stokes eigenvalue problem in the original finite element space;
(2)\ solve an additional Stokes source problem in an augmented space
using the previous obtained eigenvalue multiplying the corresponding
eigenfunction as the load vector. This
 method can improve the convergence rate of the eigenpair
 approximations with relative inexpensive computation because we replace the solution of the
 eigenvalue problem by an additional source problem  on a finer mesh or in a higher order
 finite element space.

An outline of the paper goes as follows. In Section 2, we introduce
the application of LPS method for Stokes eigenvalue problem.
 The corresponding error estimate is given in section 3.
 Section 4 is devoted to deriving the postprocessing technique and analyze its
efficiency. In Section 5, we propose a practical computational
algorithm to implement the postprocessing method. In Section 6, we
give two numerical results to confirm the theoretical analysis. Some
concluding remarks are given in the last section.

\section{Discretizations of the Stokes eigenvalue problem}
In this paper, we use the standard notations (\cite{BrennerScott,
BrezziFortin, Ciarlet}) for the Sobolev spaces $H^m(\Omega)$
(standard interpolation spaces for real number $m$) and their
associated inner products $(\cdot,\cdot)_m$, norms $\|\cdot\|_m$ and
seminorms $|\cdot |_m$ for $m\geq0$. The Sobolev space $H^0(\Omega)$
coincides with $L^2(\Omega)$, in which case the norm and inner
product are denoted by $\|\cdot\|$ and $(\cdot,\cdot)$,
respectively. In addition, denoted by $L^2_0(\Omega)$ the subspace
of $L^2(\Omega)$ that consists of functions on $L^2(\Omega)$ having
mean value zero. We also use the vector valued functions
$(H^m(\Omega))^2$ just as \cite{BrezziFortin} and
\cite{GiraultRaviart}.

The corresponding weak form of (\ref{stokesproblem}) is:

Find $({\bf u}, p,\lambda)\in {\bf V}\times Q\times\mathbb{R}$ such
that $r({\mathbf u}, {\mathbf u})=1$ and
\begin{equation}\label{wstokesproblem}
\left\{ \begin{array}{rcl}
 a({\mathbf u},{\mathbf v})-b({\mathbf v}, p)&=&\lambda r({\mathbf u}, {\mathbf v})\ \ \
 \forall {\mathbf v}\in {\bf V},\\
 b({\mathbf u}, q)&=&0\ \ \ \ \ \ \ \ \ \ \ \ \forall q\in Q,
\end{array}\right.
\end{equation}
where ${\bf V}=(H_0^1(\Omega))^2$, $Q=L_0^2(\Omega)$ and
\begin{eqnarray*}
a({\mathbf u}, {\mathbf v})&=&\int_{\Omega}\nabla{\mathbf
u}\nabla{\bf
v}d\Omega,\\
b({\mathbf v}, p)&=&\int_{\Omega}\nabla\cdot{\mathbf v}pd\Omega,\\
r({\mathbf u}, {\mathbf v})&=&\int_{\Omega}{\mathbf u}{\bf
v}d\Omega.
\end{eqnarray*}
From \cite{BabuskaOsborn},  we know eigenvalue problem
(\ref{wstokesproblem}) has an eigenvalue sequence $\{\lambda_j\}$:
$$0<\lambda_1\leq\lambda_2\leq \cdots \leq\lambda_k\leq\cdots, \ \ \lim_{k\rightarrow\infty}\lambda_k=\infty,$$
and the associated eigenfunctions
$$({\mathbf u}_1, p_1), ({\mathbf u}_2, p_2), \cdots, ({\mathbf u}_k, p_k), \cdots,$$
where $r({\mathbf u}_i, {\mathbf u}_j)=\delta_{ij}$.

For the aim of analysis, we define the bilinear form as 
\begin{eqnarray}\label{bilinear_A}
A((\mathbf u,p);(\mathbf v,q))&=&a(\mathbf u,\mathbf v)-b(\mathbf v,p)+b(\mathbf u,q).
\end{eqnarray}
For simplicity, we only consider the simple eigenvalues in this
paper. We know that $a(\cdot,\cdot), b(\cdot,\cdot)$ and
$s(\cdot,\cdot)$ have the following properties
(\cite{GiraultRaviart}):
\begin{align}
a({\mathbf u},{\mathbf v})&\leq \|\mathbf u\|_1\|\mathbf v\|_1,\label{2.2}\\
a({\mathbf u}, {\mathbf u})&\geq C\|\mathbf u\|_1^2,\\
r({\mathbf u}, {\mathbf v})&\leq C\|{\mathbf u}\|_0\|{\mathbf v}\|_0,\\
r({\mathbf u}, {\mathbf u})&\geq C\|{\mathbf u}\|_0^2,\\
 \sup_{0\neq{\bf
v}\in{\bf V}}\frac{b({\mathbf v},q)}{\|{\bf
v}\|_1}&\geq C\|q\|_0,\\
\|{\mathbf u}\|_1+\|p\|_0&\leq C\sup_{0\neq({\mathbf v},q)\in{\bf
V}\times Q}\frac{A((\mathbf u,p);(\mathbf v,q))}{\|{\mathbf v}\|_1+\|q\|_0},\label{2.5}
\end{align}
where $C>0$. In this paper, $C$ denotes constant independent of the
mesh size $h$ and sometimes depends on the eigenvalue $\lambda$ and
may be different values at its different occurrence.

For the eigenvalue, there exists the following Rayleigh quotient
expression
\begin{align}
\lambda=\frac{a({\mathbf u}, {\mathbf u})}{r({\mathbf u}, {\bf
u})}.
\end{align}

\subsection{Local projection stabilization}
In this section, we consider equal order interpolations stabilized
by the local projection method in its one-level variant as developed
in \cite{GanesanMatthiesTobiska, MatthiesSkrzypaczTobiska_Oseen}.
For the two-level approach we refer to
\cite{BeckerBraack_stokes,BB06,NW09}. Let $V_h$ denote a scalar
finite element space of continuous, piecewise polynomials over
$\mathcal{T}_h$. The spaces for approximating velocity and pressure
are given by $\mathbf V_h:=V_h^2\cap \mathbf V$ and $Q_h:=V_h\cap
Q$. The discrete problem of our stabilized method is:
\medskip

 Find $(\mathbf u_h,p_h)\in \mathbf V_h\times Q_h$ such that
\begin{equation}\label{dwstokesproblem}
\left\{ \begin{array}{rcl}
a({\mathbf u}_h, {\mathbf v})-b({\mathbf
v}, p_h)&=&\lambda_h r({\mathbf u}_h,{\mathbf v})\ \ \ \ \forall{\bf
v}\in {\bf
V}_h,\\
b({\mathbf u}_h,q)+S_h(p_h,q_h)&=&0\ \ \ \ \ \ \ \ \ \quad\quad\
\forall q\in Q_h,
\end{array}\right.
\end{equation}
where the stabilization term with user-chosen parameters $\alpha_K$
is given by
\begin{align}\label{stabilization_term}
S_h(p,q)=\sum_{K\in\mathcal{T}_h}\alpha_K(\kappa_h\nabla
p,\kappa_h\nabla q)_K.
\end{align}
Here, the fluctuation operator $\kappa_h:L^2(\Omega)\rightarrow
L^2(\Omega)$ acting componentwise is defined as follows. Let
$P_s(K)$ denote the set of all polynomials of degree less than or
equal to $s$ and let $D_h(K)$ be a finite dimensional space on the
cell $K\in{\mathcal T}_h$  with $P_s(K)\subset D_h(K)$. We extend
the definition by allowing $P_{-1}(K)=D_h(K)=\{0\}$. We introduce
the associated global space of discontinuous finite elements
$$D_h:=\bigoplus_{K\in{\mathcal T}_h}D_h(K)$$
and the local $L^2(K)$-projection $\pi_K:L^2(K)\rightarrow D_h(K)$
generating the global projection $\pi_h:L^2(\Omega)\rightarrow D_h$
by
$$ (\pi_hw)\big|_K:=\pi_K(w|_K)\qquad\forall K\in{\mathcal T}_h,\; \forall
w\in L^2(\Omega).$$ The fluctuation operator
$\kappa_h:L^2(\Omega)\rightarrow L^2(\Omega)$ used in
\eqref{stabilization_term} is given by $\kappa_h:=id-\pi_h$ where
$id:L^2(\Omega)\rightarrow L^2(\Omega)$ is the identity on
$L^2(\Omega)$.

In order to study the convergence properties of this method for Stokes eigenvalue problem, we
introduce the bilinear form
\begin{align}\label{bilinear}
A_h((\mathbf u,p);(\mathbf v,q))&=(\nabla \mathbf u,\nabla\mathbf
v)-(p,\text{div }\mathbf v)+(q,\text{div }\mathbf u)+S_h(p,q).
\end{align}
and the mesh-dependent norm
\begin{align}\label{mesh-norm}
|||(\mathbf v,q)|||_A&:=\Bigl(|\mathbf v|_1^2+\|q\|_0^2
+\sum_{K\in\mathcal{T}_h}\alpha_K\|\kappa_h\nabla
q\|_{0,K}^2\Bigr)^{1/2}.
\end{align}

The existence and uniqueness of discrete solutions of Stokes problem
have been studied in \cite{MatthiesSkrzypaczTobiska_Oseen,
GanesanMatthiesTobiska} for different pairs $(V_h,D_h)$ of
approximation and projection spaces, respectively. Based on these
results, the existence and uniqueness of eigenvalue problem
(\ref{dwstokesproblem}) can be given similarly.

The stability and convergence properties of the LPS method (\ref{dwstokesproblem})
need the following assumptions(\cite{MatthiesSkrzypaczTobiska_Oseen, RoosStynesTobiska}).\\
\textbf{Assumption A1:} There is an interpolation operator
$i_h: H^2(\Omega)\to V_h$ such that
\begin{equation}\label{assumption_a1}
\|v-i_hv\|_{0,K}+h_K|v-i_hv|_{1,K} \le C h_k^l\|v\|_{l,\omega(K)} \ \
\end{equation}
for all $ K\in \mathcal T_h$,  $v\in H^l(\omega(K))$ and $ 1\le l \le k+1$, where $\omega(K)$
denotes a certain local neighborhood of $K$ which appears in the definition of these interpolation
 operators for non-smooth functions; see \cite{Clement, ScottZhang} for more details.\\
 \textbf{Assumption A2:} The fluctuation operator $\kappa_h$ satisfy the following approximation property
\begin{equation}\label{assumption_a2}
\|\kappa_h q\|_{0,K}\le C h^l_K|q|_{l,K} \ \ \ \forall K\in \mathcal
T_h, \:\forall q\in H^l(K), \:0\le l \le k.
\end{equation}\\
\textbf{Assumption A3:} There exists a constant $\beta_1>0$  such that for all $h>0$
\begin{equation}\label{assumption_a3}
\inf_{{q_h\in\mathcal D_h(K)}}\sup_{{v_h\in V_h(K)}} \frac{(v_h,
q_h)}{\|v_h\|_{0,K} \:\|q_h\|_{0,K}}\ge \beta_1>0
\end{equation}
is satisfied where $V_h(K) =\{v_h|_K: v_h\in V_h,\ v_h=0\ {\rm in}\ \Omega\backslash K \}$.

The assumption A1 and A3 guarantee the existence of an interpolant with the usual interpolation properties
(\ref{assumption_a1}) and the orthogonality
\begin{equation}\label{orthogonality_condition}
(v-j_hv, q_h) = 0\ \  \ \ \forall q_h\in \mathcal D_h, \forall v\in
H^2(\Omega),
\end{equation}
whereas A2 is needed to bound the consistency error \cite{RoosStynesTobiska}.
For example, in the one-level LPS assumption A1 and A2
are satisfied if we choose $(V_h,D_h)=(P_k, P_{k-1}^{\rm disc})$ continuous and discontinuous,
piecewise polynomials of
degree $r$ and $r-1$, respectively. In order to guarantee A3,  $V_h$ need to be enriched
by suitable bubble functions.
For more details about LPS method, please read the papers
\cite{MatthiesSkrzypaczTobiska_Oseen, MatthiesSkrzypaczTobiska_CDR} and the book \cite{RoosStynesTobiska}.



\begin{lemma}(\cite{GanesanMatthiesTobiska})\label{BB_Condition}
Let the assumption A1, A3, and $\alpha_K\sim h_K^2$ be fulfilled. Then, there is a positive constant
$\beta_A$ independent of $h$ such that
\begin{align}
\inf_{(\mathbf v_h,q_h)\in \mathbf V_h\times Q_h} \sup_{(\mathbf
w_h,r_h)\in \mathbf V_h\times Q_h} \frac{A_h\big((\mathbf
v_h,q_h);(\mathbf w_h,r_h)\big)}{|||(\mathbf
  v_h,q_h)|||_A\; |||(\mathbf w_h,r_h)|||_A}\ge \beta_A>0
\end{align}
holds.
\end{lemma}
Based on Lemma \ref{BB_Condition}, the discrete Stokes eigenvalue problem (\ref{dwstokesproblem})
is consistent with the continuous problem (\ref{wstokesproblem}) (\cite{GanesanMatthiesTobiska}).

\section{Convergence analysis}
In thois section, we give the convergence analysis for the eigenpair
approximation $(\mathbf u_h,p_h,\lambda_h)$ in
(\ref{dwstokesproblem}).

We know that the convergence rate of the eigenpair approximations by
the finite element methods depends on the regularities of the exact
eigenfunctions. The exact  eigenfunctions of the Stokes problem only
belong to the space $(H^1(\Omega))^2\times H^0(\Omega)$ on general
domains.
 But for the domains with smooth boundary, the exact
eigenfunctions have additional regularities. In this case we need to
use isoparametric mixed finite element methods to fit the domain
more exactly (\cite{BrennerScott} and \cite{Ciarlet}). The goals of
this paper are to use LPS method to solve the Stokes eigenvalue
problem,  and propose and analyze a postprocessing method which can
improve the convergence rate for both eigenvalue and  eigenfunction
approximations. The assumption that $\Omega$ is a convex polygonal
domain can make the expression of the main idea of this paper more
directly. But, we need to notice that this assumption limits the
regularity of the exact eigenfunctions and makes the analysis of the
convergence rates much more complicated. It is well known
(\cite{BacutaBramble, BacutaBramblePasciak, FabesKenigVerchota})
that for a given ${\bf f} \in (H^\gamma(\Omega))^2$ the solution
$({\mathbf u},p)$ of the corresponding Stokes  problem
\begin{equation}
\left\{ \begin{array}{rcl}a({\mathbf u}, {\mathbf v})-b({\bf
v},p)&=&r({\bf
f}, {\mathbf v})\ \ \ \forall {\mathbf v}\in{\bf V},\\
b({\mathbf u}, q)&=&0\ \ \quad \quad \forall q\in Q
\end{array} \right.
\end{equation}
has the following regularity (\cite{BacutaBramble, BacutaBramblePasciak, BlumRannacher, FabesKenigVerchota})
\begin{align}\label{regularity}
\|{\mathbf u}\|_{2+\gamma}+\|p\|_{1+\gamma}\leq C\|{\bf
f}\|_{\gamma}\ \ \ \forall {\bf f}\in (H^{\gamma}(\Omega))^2,
\end{align}
where $0<\gamma\leq 1$ is a parameter that depends on the largest
interior angle of $\partial\Omega$ (\cite{BacutaBramble}).

From (\ref{dwstokesproblem}), we can know the following Rayleigh
quotient for $\lambda_h$ holds
\begin{align}\label{drayleigh}
\lambda_h=\frac{a({\mathbf u}_h,{\mathbf u}_h)}{r({\mathbf
u}_h,{\mathbf u}_h)}
\end{align}
Holds.

It is also known from
\cite{BabuskaOsborn} the Stokes eigenvalue problem
(\ref{dwstokesproblem}) has eigenvalues
$$0<(\lambda_1)_h\leq(\lambda_2)_h\leq\cdots\leq(\lambda_k)_h\leq\cdots\leq(\lambda_N)_h,$$
and the corresponding eigenfunctions
$$(({\mathbf u}_1)_h, (p_1)_h), (({\mathbf u}_2)_h, (p_2)_h),\cdots,
 (({\mathbf u}_k)_h, (p_k)_h),\cdots,
(({\mathbf u}_N)_h, (p_N)_h),$$ where $r(({\mathbf u}_i)_h, ({\bf
u}_j)_h)=\delta_{ij}, 1\leq i,j\leq N$, $N$ denotes the dimension of the finite element space
 $\mathbf V_h\times Q_h$ .
 
 Let us define the compact operator $T: (L^2(\Omega))^2\rightarrow (H^1(\Omega))^2$ and the operator 
 $K: (L^2(\Omega))^2\rightarrow L^2_0(\Omega)$ by 
 \begin{eqnarray}
 A(( T\mathbf f,K\mathbf f),(\mathbf v,q))&=&r(\mathbf f,\mathbf v),\ \ \ \forall (\mathbf v,q)\in (H_0^1(\Omega))^2\times L_0^2(\Omega).
 \end{eqnarray}
Hence the eigenvalue problem (\ref{wstokesproblem}) can be written as 
\begin{eqnarray}
\lambda T\mathbf u&=&\mathbf u.
\end{eqnarray}
Let $M(\lambda_i)$ denote the eigenspace corresponding to the eigenvalue $\lambda_i$ which is defined by 
\begin{eqnarray*}
M(\lambda_i)&=&\Big\{(\mathbf w, \psi)\in (H_0^1(\Omega))^2\times L_0^2(\Omega):\ (\mathbf w, \psi)\ {\rm is\ an\ eigenfunction\ of}\ (\ref{wstokesproblem}) \nonumber\\
&&\ \ \ \ \  {\rm corresponding\ to\ } \lambda_i \ {\rm  and}\  r(\mathbf w,\mathbf w)=1\Big\}.
\end{eqnarray*}

Similarly, we also introduce  the discrete operator $T_h: (L^2(\Omega))^2\rightarrow \mathbf V_h$ and the operator
 $K_h: (L^2(\Omega))^2\rightarrow Q_h$ by
 \begin{eqnarray}
 A_h(( T_h\mathbf f,K_h\mathbf f),(\mathbf v,q))&=&r(\mathbf f,\mathbf v), 
 \ \ \ \forall (\mathbf v,\mathbf q)\in \mathbf V_h\times Q_h.
 \end{eqnarray}
Hence the operator form of the discrete eigenvalue problem (\ref{dwstokesproblem}) is 
\begin{eqnarray}
\lambda_h T_h\mathbf u_h&=&\mathbf u_h.
\end{eqnarray}

In \cite{GanesanMatthiesTobiska}, the convergence result of LPS method for Stokes problems has
been given. Combining abstract spectral approximation results from \cite{BabuskaOsborn}, we can give the
 convergence results for the Stokes eigenvalue problem by LPS method.
 The eigenvalue approximation
$\lambda_h$ and the corresponding eigenfunction approximation
$({\mathbf u}_h, p_h)$ have the following error bounds
(\cite{BabuskaOsborn1989, Osborn, FabesKenigVerchota,
GanesanMatthiesTobiska, MercierOsbornRappazRaviart,
GiraultRaviart}):
\begin{eqnarray}
\|\mathbf u-\mathbf u_h\|_1 &\leq&C\|(T-T_h)|_{M(\lambda)}\|_1.
\end{eqnarray}

%

In order to do the analysis for the postprocessing in the following sections, we also need the
 convergence result for
the eigenfunction approximation $\mathbf u_h$ in $H^{-1}$-norm. For
this aim, based on the result in \cite{BabuskaOsborn1989},  we first
need to
 use the duality argument to get $H^{-1}$-norm
error estimate for the finite element projection and the process is
similar to the one in the paper \cite{GanesanMatthiesTobiska} for
the $L^2$-norm error estimate.

The finite element projection $(R_h\mathbf u, R_hp)$ denotes the finite element solution of the
 following Stokes problem:

Find $(R_h(\mathbf u, p), G_h(\mathbf u,p))\in \mathbf V_h\times Q_h$ such that
\begin{eqnarray}\label{fem_projection}
A_h((R_h(\mathbf u, p), G_h(\mathbf u,p));(\mathbf v,q))=A((\mathbf u,p);(\mathbf v,q))\ \ \ \forall (\mathbf v,q)\in \mathbf V_h\times Q_h.
\end{eqnarray}
\begin{eqnarray}
T_h = R_h(T, K), && K_h = G_h(T, K). 
\end{eqnarray}
\begin{eqnarray}
\|\mathbf u-\mathbf u_h\|_1 &\leq&C\|(T-R_h(T, K)|_{M(\lambda)}\|_1.
\end{eqnarray}

From the definition, we have the orthogonal relation
\begin{eqnarray}\label{orognal_relation_fem_projection}
A_h((\mathbf u-R_h(\mathbf u,p),p-G_h(\mathbf u,p));(\mathbf v,q))&=&S_h(p,q)\nonumber\\ 
&&\ \ \forall (\mathbf v,q)\in \mathbf V_h\times Q_h.
\end{eqnarray}
\begin{eqnarray}
&&|||(\mathbf v_h-R_h(\mathbf u,p),q_h-G_h(\mathbf u,p))|||_A\nonumber\\
&\leq& \frac{1}{\beta_A}\sup_{0\neq(\mathbf w_h,\psi_h)\in \mathbf V_h\times Q_h}\frac{A_h((\mathbf v_h- R_h(\mathbf u,p),q_h-G_h(\mathbf u,p));(\mathbf w_h,\psi_h))}{|||(\mathbf w_h,\psi_h)|||_A}\nonumber\\
&\leq&\frac{1}{\beta_A}\sup_{0\neq(\mathbf w_h,\psi_h)\in \mathbf V_h\times Q_h}\frac{A_h((\mathbf v_h-\mathbf u,q_h-p);(\mathbf w_h,\psi_h))}{|||(\mathbf w_h,\psi_h)|||_A}\nonumber\\
&&+ \frac{1}{\beta_A}\sup_{0\neq(\mathbf w_h,\psi_h)\in \mathbf V_h\times Q_h}\frac{A_h((\mathbf u- R_h(\mathbf u,p),p-G_h(\mathbf u,p));(\mathbf w_h,\psi_h))}{|||(\mathbf w_h,\psi_h)|||_A}\nonumber\\
&\leq&\frac{C}{\beta_A}|||(\mathbf u-\mathbf v_h,p-q_h)|||_A+ 
\frac{1}{\beta_A}\sup_{0\neq(\mathbf w_h,\psi_h)\in \mathbf V_h\times Q_h}\frac{S_h(p,\psi_h)}{|||(\mathbf w_h,\psi_h)|||_A}\nonumber\\
&\leq&\frac{C}{\beta_A}|||(\mathbf u-\mathbf v_h,p-q_h)|||_A+ \frac{1}{\beta_A}S_h^{1/2}(p,p)
\end{eqnarray}
\begin{eqnarray}
&&|||(\mathbf u-R_h(\mathbf u,p),p-G_h(\mathbf u,p))|||_A\nonumber\\
 &\leq& 
|||(\mathbf u-\mathbf v_h,p-q_h)|||_A +
|||(\mathbf v_h-R_h(\mathbf u,p),q_h-G_h(\mathbf u,p))|||_A\nonumber\\
&\leq& C|||(\mathbf u-\mathbf v_h,p-q_h)|||_A+ C S_h^{1/2}(p,p).
\end{eqnarray}
The arbitrariness of $(\mathbf v_h,q_h)$ leads to the following
\begin{eqnarray}
&&|||(\mathbf u-R_h(\mathbf u,p),p-G_h(\mathbf u,p))|||_A \nonumber\\
& \leq&
 C\big(\inf_{(\mathbf v_h,q_h)\in \mathbf V_h\times Q_h}|||(\mathbf u-\mathbf v_h,p-q_h)|||_A
  + S_h^{1/2}(p,p)\big).
\end{eqnarray} 

\begin{eqnarray}
\|\mathbf u-\mathbf u_h\|_1 &\leq& C\delta_h(\lambda),
\end{eqnarray}
where $\delta_h(\lambda)$ is defined by 
\begin{eqnarray}
\delta_h(\lambda) := \sup_{(\mathbf w,\psi)\in M(\lambda)}\big(\inf_{(\mathbf v_h,q_h)\in \mathbf V_h\times Q_h}|||(\mathbf w-\mathbf v_h,\psi - q_h)|||_A + S_h^{1/2}(\psi,\psi)\big).
\end{eqnarray}

We choose $\mathbf g\in (H_0^1(\Omega))^2$ such that $\|\mathbf
g\|_1=1$ and
\begin{eqnarray*}
\|\mathbf u-R_h(\mathbf u,p)\|_{-1} &=& r(\mathbf u-R_h(\mathbf u,p), \mathbf
g).
\end{eqnarray*}
Then we define a duality problem corresponding to $\mathbf g$:

Find $(\mathbf u_{g}, p_g)\in \mathbf V\times Q$ such that
\begin{equation}\label{dual_equation}
\left\{\begin{array}{rcl}
a(\mathbf v, \mathbf u_g)-b(\mathbf
v,p_g)&=&r(\mathbf v, \mathbf
g)\ \ \ \ \forall \mathbf v\in\mathbf V,\\
b(\mathbf u_g, q)&=&0\quad\quad\quad\ \ \forall q\in Q.
\end{array}
\right.
\end{equation}
Combination of (\ref{orognal_relation_fem_projection}) and
(\ref{dual_equation}) derives the following estimate
\begin{eqnarray*}
r(\mathbf u-R_h(\mathbf u,p), \mathbf g) &=& a(\mathbf u-R_h(\mathbf u,p),
\mathbf
u_g)-b(\mathbf u-R_h(\mathbf u,p), p_g)\nonumber\\
&=&a(\mathbf u-R_h(\mathbf u, p),\mathbf
u_g-\mathbf v_h)-b(\mathbf u-R_h(\mathbf u,p), p_g-q_h)\nonumber\\
&&-b(\mathbf u_g-\mathbf v_h, p-G_h(\mathbf u,p))-S_h(G_h(\mathbf u,p),q_h)\nonumber\\
&=&a(\mathbf u-R_h(\mathbf u, p),\mathbf
u_g-\mathbf v_h)-b(\mathbf u-R_h(\mathbf u,p), p_g-q_h)\nonumber\\
&&-b(\mathbf u_g-\mathbf v_h, p-G_h(\mathbf u,p))-S_h(p-G_h(\mathbf u,p),p_g-q_h)\nonumber\\
&&+S_h(p-G_h(\mathbf u,p),p_g)+S_h(p,p_g-q_h)-S_h(p,p_g).
\end{eqnarray*}
Choosing $(\mathbf v_h, q_h)\in \mathbf V_h\times Q_h$ as an
interpolant of $(\mathbf u_g,p_g)$, we obtain
\begin{eqnarray*}
&&|r(\mathbf u-R_h(\mathbf u,p), \mathbf g)|\nonumber\\
&\leq& C\big(\|\mathbf u-R_h(\mathbf
u,p)\|_1+\|p-G_h(\mathbf u, p)\|_0)(\|\mathbf u_g-\mathbf
v_h\|_1+\|p_g-q_h\|_0\big)\nonumber\\
&&+|S_h(G_h(\mathbf u,p),q_h)|\nonumber\\
&\leq& C \big(|||(\mathbf u-R_h(\mathbf u,p), p-G_h(\mathbf u,p))|||_A+S_h^{1/2}(p,p)\big)\nonumber\\
&&\ \ \ \ 
\big(|||(\mathbf u_g-\mathbf v_h, p_g-q_h)|||_A+S_h^{1/2}(p_g,p_g)\big).
\end{eqnarray*}
In particular, when $\Omega$ is smooth, we have the regularity
estimate
\begin{eqnarray*}
\|\mathbf u_g\|_3+\|p_g\|_2\leq C\|\mathbf g\|_1,
\end{eqnarray*}

\begin{eqnarray*}
\|\mathbf u-R_h\mathbf u\|_{-1}&\leq& C \eta_h\delta_h,
\end{eqnarray*}
where 
\begin{eqnarray}
\eta_h=\sup_{\|\mathbf g\|_1=1}\inf_{(\mathbf v,q)\in\mathbf V_h\times Q_h}\big(
|||(T\mathbf g-\mathbf v,K\mathbf g-q)|||_A+S_h^{1/2}(K\mathbf g,K\mathbf g)\big).
\end{eqnarray}



\section{One correction step}
In this section, we present a type of correction step to improve the
accuracy of the current eigenvalue and eigenfunction approximations.
This correction method contains solving an auxiliary source problem
in the finer finite element space and an eigenvalue problem on the
coarsest finite element space. For simplicity of notation, we set
$(\lambda,u)=(\lambda_i,u_i)\ (i=1,2,\cdots,k,\cdots)$  and
$(\lambda_h, u_h)=(\lambda_{i,h},u_{i,h})\ (i=1,2,\cdots,N_h)$ to
denote an eigenpair of problem (\ref{weak_problem}) and
(\ref{weak_problem_Discrete}), respectively.

To derive our method, we need first to introduce  the error
expansions of the eigenvalues by the Rayleigh quotient formula. It
is well known that there have been the Rayleigh quotient error
expansions for the eigenvalues of the second  order elliptic
problems (\cite{LinYan}).

\begin{theorem}\label{TS.7}
Assume $({\mathbf u}, p, \lambda)$ is the true solution of the
Stokes eigenvalue problem (\ref{wstokesproblem}),  $0\neq {\mathbf
w}\in (H_0^1(\Omega))^2$ and $\psi\in L_0^2(\Omega)$ satisfy
\begin{align}\label{orth}
b({\mathbf w}, \psi)+S_h(\psi,\psi)=0.
\end{align}
Let us define
\begin{align}\label{rayleighw}
\hat{\lambda}=\frac{a({\mathbf w},{\mathbf w})}{r({\mathbf w},{\bf
w})}.
\end{align}
Then, we have
\begin{align}\label{rayexpan}
\hat{\lambda}-\lambda&=\frac{a({\mathbf w}-{\mathbf u},{\mathbf
w}-{\bf u})+2b({\mathbf w}-{\mathbf u},p-\psi)-\lambda r({\bf
w}-{\mathbf u},{\mathbf w}-{\mathbf u})}{r({\mathbf w},{\mathbf
w})}\nonumber\\
&\quad\quad\quad
 -\frac{2S_h(\psi,\psi)}{r({\mathbf w},{\mathbf w})}.
\end{align}
If the condition (\ref{orth}) is changed to be
\begin{align}\label{orth2}
 b(\mathbf w,\varphi)&=0,
\end{align}
the expansion for $\hat{\lambda}-\lambda$ should be
\begin{align}\label{rayexpan2}
\hat{\lambda}-\lambda&=\frac{a({\mathbf w}-{\mathbf u},{\mathbf
w}-{\bf u})+2b({\mathbf w}-{\mathbf u},p-\psi)-\lambda r({\bf
w}-{\mathbf u},{\mathbf w}-{\mathbf u})}{r({\mathbf w},{\mathbf
w})}.
\end{align}
\end{theorem}
\begin{proof}
From (\ref{wstokesproblem}), (\ref{dwstokesproblem}),
(\ref{drayleigh}), (\ref{orth}), (\ref{rayleighw}) and direct
computation, we have
\begin{eqnarray*}
\hat{\lambda}-\lambda&=&\frac{a({\mathbf w},{\mathbf w})-\lambda
r({\bf
w},{\mathbf w})}{r({\mathbf w},{\mathbf w})}\\
&=&\frac{a({\mathbf w}-{\mathbf u},{\mathbf w}-{\mathbf u})+2a({\bf
w},{\bf u})-a({\mathbf u},{\mathbf u})-\lambda r({\mathbf w},{\bf
w})}{r({\mathbf w},{\bf
w})} \\
&=&\frac{a({\mathbf w}-{\mathbf u},{\mathbf w}-{\mathbf u})+2\lambda
r({\mathbf w},{\mathbf u})+2b({\mathbf w}, p)-\lambda r({\mathbf
u},{\bf u})-\lambda r({\mathbf w},{\bf
w})}{r({\mathbf w},{\mathbf w})} \\
&=&\frac{a({\mathbf w}-{\mathbf u},{\mathbf w}-{\mathbf u})-\lambda
r({\mathbf w}-{\bf
u},{\mathbf w}-{\mathbf u})+2b({\mathbf w}, p)}{r({\mathbf w},{\mathbf w})}\\
&=&\frac{a({\mathbf w}-{\mathbf u},{\mathbf w}-{\mathbf u})-\lambda
r({\mathbf w}-{\bf
u},{\mathbf w}-{\mathbf u})+2b({\mathbf w}, p-\psi)-2S_h(\psi,\psi)}{r({\mathbf w},{\mathbf w})}\\
&=&\frac{a({\mathbf w}-{\mathbf u},{\mathbf w}-{\mathbf u})-\lambda
r({\mathbf w}-{\mathbf u},{\mathbf w}-{\mathbf u})+2b({\bf
w}-{\mathbf u}, p-\psi)-2S_h(\psi,\psi)}{r({\mathbf w},{\mathbf
w})}.
\end{eqnarray*}
This is the desired result (\ref{rayexpan}) and the expansion of (\ref{rayexpan2}) can be prooved
similarly.
\end{proof}

Assume we have obtained an eigenpair approximation
$(\lambda_{h_1},\mathbf u_{h_1}, p_{h_1})\in\mathcal{R}\times \mathbf V_{h_1}\times Q_h$. Now we
introduce a type of correction step to improve the accuracy of the
current eigenpair approximation $(\lambda_{h_1},\mathbf u_{h_1}, p_{h_1})$. Let
$\mathbf V_{h_2}\times Q_{h_2}\subset (H_0^1(\Omega))^2\times L_0^2(\Omega)$ be a finer finite element space such that
$\mathbf V_{h_1}\times Q_{h_1}\subset \mathbf V_{h_2}\times Q_{h_2}$. Based on this finer finite element space,
we define the following correction step.

\begin{algorithm}\label{Correction_Step}
One Correction Step

\begin{enumerate}
\item Define the following auxiliary source problem:

Find $\tilde{u}_{h_2}\in V_{h_2}$ such that
\begin{eqnarray}\label{aux_problem}
a(\widetilde{u}_{h_2},v_{h_2})&=&\lambda_{h_1}b(u_{h_1},v_{h_2}),\ \
\ \forall v_{h_2}\in V_{h_2}.
\end{eqnarray}
Solve this equation to obtain a new eigenfunction approximation
$\tilde{u}_{h_2}\in V_{h_2}$.
\item  Define a new finite element
space $V_{H,h_2}=V_H+{\rm span}\{\widetilde{u}_{h_2}\}$ and solve
the following eigenvalue problem:

Find $(\lambda_{h_2},u_{h_2})\in\mathcal{R}\times V_{H,h_2}$ such
that $b(u_{h_2},u_{h_2})=1$ and
\begin{eqnarray}\label{Eigen_Augment_Problem}
a(u_{h_2},v_{H,h_2})&=&\lambda_{h_2} b(u_{h_2},v_{H,h_2}),\ \ \
\forall v_{H,h_2}\in V_{H,h_2}.
\end{eqnarray}
\end{enumerate}
Define $(\lambda_{h_2},u_{h_2})={\it
Correction}(V_H,\lambda_{h_1},u_{h_1},V_{h_2})$.
\end{algorithm}
\begin{theorem}\label{Error_Estimate_One_Correction_Theorem}
Assume the current eigenpair approximation
$(\lambda_{h_1},u_{h_1})\in\mathcal{R}\times V_{h_1}$ has the
following error estimates
\begin{eqnarray}
\|u-u_{h_1}\|_a &\lesssim &\varepsilon_{h_1}(\lambda),\label{Error_u_h_1}\\
\|u-u_{h_1}\|_{-a}&\lesssim&\eta_a(H)\|u-u_{h_1}\|_a,\label{Error_u_h_1_nagative}\\
|\lambda-\lambda_{h_1}|&\lesssim&\varepsilon_{h_1}^2(\lambda).\label{Error_Eigenvalue_h_1}
\end{eqnarray}
Then after one correction step, the resultant approximation
$(\lambda_{h_2},u_{h_2})\in\mathcal{R}\times V_{h_2}$ has the
following error estimates
\begin{eqnarray}
\|u-u_{h_2}\|_a &\lesssim &\varepsilon_{h_2}(\lambda),\label{Estimate_u_u_h_2}\\
\|u-u_{h_2}\|_{-a}&\lesssim&\eta_a(H)\|u-u_{h_2}\|_a,\label{Estimate_u_h_2_Nagative}\\
|\lambda-\lambda_{h_2}|&\lesssim&\varepsilon_{h_2}^2(\lambda),\label{Estimate_lambda_lambda_h_2}
\end{eqnarray}
where
$\varepsilon_{h_2}(\lambda):=\eta_a(H)\varepsilon_{h_1}(\lambda)+\varepsilon_{h_1}^2(\lambda)+\delta_{h_2}(\lambda)$.
\end{theorem}
\begin{proof}
From problems (\ref{Projection_Problem}), (\ref{weak_problem}) and
(\ref{aux_problem}), and (\ref{Error_u_h_1}),
(\ref{Error_u_h_1_nagative}) and (\ref{Error_Eigenvalue_h_1}), the
following estimate holds
\begin{eqnarray*}
\|\widetilde{u}_{h_2}-P_{h_2}u\|_a^2&\lesssim
&a(\widetilde{u}_{h_2}-P_{h_2}u,\widetilde{u}_{h_2}-P_{h_2}u)=b(\lambda_{h_1}u_{h_1}-\lambda
u,\widetilde{u}_{h_2}-P_{h_2}u)\nonumber\\
& \lesssim &\|\lambda_{h_1}u_{h_1}-\lambda
u\|_{-a}\|\widetilde{u}_{h_2}-P_{h_2}u\|_a\nonumber\\
&\lesssim & (|\lambda_{h_1}-\lambda|\|u_{h_1}\|_{-a}+\lambda
\|u_{h_1}-u\|_{-a})\|\widetilde{u}_{h_2}-P_{h_2}u\|_a\nonumber\\
&\lesssim &\big
(\varepsilon_{h_1}^2(\lambda)+\eta_a(H)\varepsilon_{h_1}(\lambda)\big)\|\widetilde{u}_{h_2}-P_{h_2}u\|_a.
\end{eqnarray*}
Then we have
\begin{eqnarray}\label{Estimate_u_tilde_u_h_2}
\|\widetilde{u}_{h_2}-P_{h_2}u\|_a &\lesssim
&\varepsilon_{h_1}^2(\lambda)+\eta_a(H)\varepsilon_{h_1}(\lambda).
\end{eqnarray}
Combining (\ref{Estimate_u_tilde_u_h_2}) and
 the error estimate of finite element projection
\begin{eqnarray*}
\|u-P_{h_2}u\|_a &\lesssim&\delta_{h_2}(\lambda),
\end{eqnarray*}
we have
\begin{eqnarray}\label{Error_tilde_u_h_2_u}
\|\widetilde{u}_{h_2}-u\|_{a}&\lesssim&\varepsilon_{h_1}^2(\lambda)+\eta_a(H)\varepsilon_{h_1}(\lambda)
+\delta_{h_2}(\lambda).
\end{eqnarray}
Now we come to estimate the eigenpair solution
$(\lambda_{h_2},u_{h_2})$ of problem (\ref{Eigen_Augment_Problem}).
Based on the error estimate theory of eigenvalue problem by finite
element method (\cite{Babuska2,BabuskaOsborn}), the following
estimates hold
\begin{eqnarray}\label{Error_u_u_h_2}
\|u-u_{h_2}\|_a&\lesssim& \sup_{w\in M(\lambda)}\inf_{v\in
V_{H,h_2}}\|w-v\|_a\lesssim \|u-\widetilde{u}_{h_2}\|_a,
\end{eqnarray}
and
\begin{eqnarray}\label{Error_u_u_h_2_Negative}
\|u-u_{h_2}\|_{-a}&\lesssim&\widetilde{\eta}_a(H)\|u-u_{h_2}\|_a,
\end{eqnarray}
where
\begin{eqnarray}\label{Eta_a_h_2}
\widetilde{\eta}_a(H)&=&\sup_{f\in V,\|f\|_a=1}\inf_{v\in
V_{H,h_2}}\|Tf-v\|_a \leq \eta_a(H).
\end{eqnarray}
From (\ref{Error_tilde_u_h_2_u}), (\ref{Error_u_u_h_2}),
(\ref{Error_u_u_h_2_Negative}) and (\ref{Eta_a_h_2}), we can obtain
(\ref{Estimate_u_u_h_2}) and (\ref{Estimate_u_h_2_Nagative}). The
estimate (\ref{Estimate_lambda_lambda_h_2}) can be derived by
Theorem \ref{Rayleigh_Quotient_error_theorem} and
(\ref{Estimate_u_u_h_2}).
\end{proof}

If the eigenpair approximation $({\mathbf u}_h, p_h, \lambda_h)$ of
the Stokes eigenvalue problem (\ref{wstokesproblem}) has been
obtained, we define the following Stokes source problem:

 Find $(\tilde{\bf
u}, \tilde{p})\in {\bf V}\times Q$ such that
\begin{equation}\label{auxstoke}
\left\{\begin{array}{rcl} a(\tilde{\mathbf u},{\mathbf v})-b({\bf
v}, \tilde{p})&=&\lambda_hr({\mathbf u}_h, {\mathbf v})\ \ \ \
\forall {\bf
v}\in {\bf V},\\
b(\tilde{\mathbf u}, q)&=&0\ \quad \quad \quad \quad \quad \forall
q\in Q.
\end{array}
\right.
\end{equation}
We also define the following Rayleigh  quotient formula for the solution $(\tilde{\mathbf u}, \tilde{p})$
\begin{align}\label{defofll}
\tilde{\lambda}=\frac{a(\tilde{\mathbf u},\tilde{\bf
u})}{r(\tilde{\mathbf u},\tilde{\mathbf u})}.
\end{align}

For the eigenpair $(\tilde{\mathbf u}, \tilde{p}, \tilde{\lambda})$, we can give the
following error estimate.

\begin{theorem}\label{Ts.8}
Assume  $( {\mathbf u}, p, \lambda)$
 is the true solution of the Stokes eigenvalue problem (\ref{wstokesproblem}), $({\mathbf u}_h,p_h, \lambda_h)$
is the corresponding finite element solution of the discrete Stokes
eigenvalue problem (\ref{dwstokesproblem}),  $(\tilde{\bf
u},\tilde{p})$ is the true solution of problem (\ref{auxstoke}) and
$\tilde{\lambda}$ is defined by (\ref{defofll}). Then we have the
following estimates
\begin{align}
\|{\mathbf u}-\tilde{\mathbf u}\|_1+\|p-\tilde{p}\|_0&\leq C(\|{\bf
u}-{\bf
u}_h\|_{-1}+|\lambda-\lambda_h|),\label{errorofpostsolution}\\
 |\tilde{\lambda}-\lambda|&\leq C(\|{\mathbf u}-{\bf
u}_h\|_{-1}^2+|\lambda-\lambda_h|^2).\label{errorofposteigenvalue}
\end{align}
\end{theorem}
\begin{proof}
First from Stokes eigenvalue problem (\ref{wstokesproblem}) and
Stokes problem (\ref{auxstoke}), we have
\begin{align}
&a(\tilde{\mathbf u}-{\mathbf u}, {\mathbf v})+b({\bf
v},\tilde{p}-p)+b(\tilde{\mathbf u}-{\mathbf u},q)\nonumber\\
= &r(\lambda_h{\bf
u}_h-\lambda {\mathbf u},{\mathbf v})\nonumber\\
= &\lambda_hr({\mathbf u}_h-{\mathbf u}, {\bf
v})+(\lambda_h-\lambda)r({\mathbf u},{\mathbf v})\nonumber\\
\leq& C(\|{\mathbf u}_h-{\mathbf
u}\|_{-1}+|\lambda_h-\lambda|)\|{\bf v}\|_1.
\end{align}
Then, from (\ref{2.5}), we have
\begin{align}\label{errorofpostproblem}
\|\tilde{\mathbf u}-{\mathbf u}\|_1+\|\tilde{p}-p\|_0&\leq
\sup_{0\neq({\mathbf v},q)\in{\bf V}\times Q}\frac{a(\tilde{\mathbf
u}-{\mathbf u}, {\mathbf v})+b({\mathbf
v},\tilde{p}-p)+b(\tilde{\mathbf u}-{\bf
u},q)}{\|{\mathbf v}\|_1+\|q\|_0}\nonumber\\
&\leq C(\|{\mathbf u}_h-{\mathbf u}\|_{-1}+|\lambda_h-\lambda|).
\end{align}
From (\ref{errorofpostproblem}) and the Rayleigh quotient expansion
(\ref{rayexpan2}),
 we obtain (\cite{ChenJiaXIe})
\begin{eqnarray*}
\tilde{\lambda}-\lambda&\leq&C(\|\tilde{\mathbf u}-{\bf
u}\|_1^2+\|\tilde{\mathbf u}-{\mathbf u}\|_1\|\tilde{p}-p\|_0)\\
&\leq &C(\|{\mathbf u}_h-{\mathbf u}\|_{-1}+|\lambda_h-\lambda|)^2\\
&\leq&C(\|{\mathbf u}_h-{\mathbf u}\|_{-1}^2+|\lambda-\lambda_h|^2).
\end{eqnarray*}
So the proof is complete.
\end{proof}
Based on the result of the convergence rate for the eigenpair
approximation, we can obtain the error estimates:

For the smooth domain, from (\ref{2.23})-(\ref{2.24}) and (\ref{-1_norm_k=1})-(\ref{-1_norm_k>=2})
\begin{align}
\|\tilde{\mathbf u}-{\mathbf u}\|_1+\|\tilde{p}-p\|_0&\leq Ch^2\ \ \ \ \ {\rm for}\ k=1,\\
|\tilde{\lambda}-\lambda|&\leq Ch^4\ \ \ \ \ {\rm for}\ k=1,\\
\|\tilde{\mathbf u}-{\mathbf u}\|_1+\|\tilde{p}-p\|_0&\leq Ch^{k+2}\ \ \ \ \ \ {\rm for}\ k\geq2,\\
|\tilde{\lambda}-\lambda|&\leq Ch^{2k+4}\ \ \ \ \ {\rm for}\ k\geq2.
\end{align}
For the convex polygonal domain, from (\ref{2.20s})-(\ref{2.21s}) and (\ref{-1_norm_s}),
we have
\begin{align}
\|\tilde{\mathbf u}-{\mathbf u}\|_1+\|\tilde{p}-p\|_0&\leq Ch^{2s},\\
|\tilde{\lambda}-\lambda|&\leq Ch^{4s}.
\end{align}
 This means that $(\tilde{\mathbf u},
\tilde{p}, \tilde{\lambda})$ is a better approximation than $({\bf
u}_h, p_h, \lambda_h)$ of the true solution $({\mathbf u}, p,
\lambda)$ of the Stokes eigenvalue problem (\ref{wstokesproblem}).

\section{Multi-level correction scheme}
In this section, we introduce a type of multi-level correction
scheme based on the {\it One Correction Step} defined in Algorithm
\ref{Correction_Step}. This type of correction method can improve
the convergence order after each correction step which is different
from the two-grid method in \cite{XuZhou}.

\begin{algorithm}\label{Multi_Correction}
Multi-level Correction Scheme
\begin{enumerate}
\item Construct a coarse finite element space $V_H$ and solve the
following eigenvalue problem:

Find $(\lambda_H,u_H)\in \mathcal{R}\times V_H$ such that
$b(u_H,u_H)=1$ and
\begin{eqnarray}\label{Initial_Eigen_Problem}
a(u_H,v_H)&=&\lambda_Hb(u_H,v_H),\ \ \ \ \forall v_H\in V_H.
\end{eqnarray}
\item Set $h_1=H$ and construct a series of finer finite element
spaces $V_{h_2},\cdots,V_{h_n}$ such that $\eta_a(H)\gtrsim
\delta_{h_1}(\lambda)\geq \delta_{h_2}(\lambda)\geq\cdots\geq
\delta_{h_n}(\lambda)$.
\item Do $k=0,1,\cdots,n-2$\\
Obtain a new eigenpair approximation
$(\lambda_{h_{k+1}},u_{h_{k+1}})\in \mathcal{R}\times V_{h_{k+1}}$
by a correction step
\begin{eqnarray}
(\lambda_{h_{k+1}},u_{h_{k+1}})=Correction(V_H,\lambda_{h_k},u_{h_k},V_{h_{k+1}}).
\end{eqnarray}
end Do
\item Solve the following source problem:

Find $u_{h_n}\in V_{h_n}$ such that
\begin{eqnarray}\label{aux_problem_h_n}
a(u_{h_n},v_{h_n})&=&\lambda_{h_{n-1}}b(v_{h_{n-1}},v_{h_n}),\ \ \
\forall v_{h_n}\in V_{h_n}.
\end{eqnarray}
Then compute the Rayleigh quotient of $u_{h_n}$
\begin{eqnarray}\label{Rayleigh_quotient_h_n}
\lambda_{h_n} &=&\frac{a(u_{h_n},u_{h_n})}{b(u_{h_n},u_{h_n})}.
\end{eqnarray}
\end{enumerate}
Finally, we obtain an eigenpair approximation
$(\lambda_{h_n},u_{h_n})\in \mathcal{R}\times V_{h_n}$.
\end{algorithm}
\begin{theorem}
After implementing Algorithm \ref{Multi_Correction}, the resultant
eigenpair approximation $(\lambda_{h_n},u_{h_n})$ has the following
error estimate
\begin{eqnarray}
\|u_{h_n}-u\|_a &\lesssim&\varepsilon_{h_n}(\lambda),\label{Multi_Correction_Err_fun}\\
|\lambda_{h_n}-\lambda|&\lesssim&\varepsilon_{h_n}^2(\lambda),\label{Multi_Correction_Err_eigen}
\end{eqnarray}
where
$\varepsilon_{h_n}(\lambda)=\sum\limits_{k=1}^{n}\eta_a(H)^{n-k}\delta_{h_k}(\lambda)$.
\end{theorem}
\begin{proof}
From $\eta_a(H)\gtrsim\delta_{h_1}(\lambda)\geq
\delta_{h_2}(\lambda)\geq\cdots\geq \delta_{h_n}(\lambda)$ and
Theorem \ref{Error_Estimate_One_Correction_Theorem}, we have
\begin{eqnarray}
\varepsilon_{h_{k+1}}(\lambda)
&\lesssim&\eta_a(H)\varepsilon_{h_k}(\lambda)+\delta_{h_{k+1}}(\lambda),\
\ \ {\rm for}\ 1\leq k\leq n-2.
\end{eqnarray}
Then by recursive relation, we can obtain
\begin{eqnarray}\label{epsilon_n_1}
\varepsilon_{h_{n-1}}(\lambda)&\lesssim&\eta_a(H)\varepsilon_{h_{n-2}}(\lambda)
+\delta_{h_{n-1}}(\lambda)\nonumber\\
&\lesssim&\eta_a(H)^2\varepsilon_{h_{n-3}}(\lambda)+
\eta_a(H)\delta_{h_{n-2}}(\lambda)+\delta_{h_{n-1}}(\lambda)\nonumber\\
&\lesssim&\sum\limits_{k=1}^{n-1}\eta_a(H)^{n-k-1}\delta_{h_k}(\lambda).
\end{eqnarray}
Based on the proof in Theorem
\ref{Error_Estimate_One_Correction_Theorem} and (\ref{epsilon_n_1}),
the final eigenfunction approximation $u_{h_n}$ has the error
estimate
\begin{eqnarray}\label{Error_u_h_n_Multi_Correction}
\|u_{h_n}-u\|_a&\lesssim&\varepsilon_{h_{n-1}}^2(\lambda)+\eta_a(H)\varepsilon_{h_{n-1}}(\lambda)
+\delta_{h_n}(\lambda)\nonumber\\
&\lesssim& \sum_{k=1}^n\eta_a(H)^{n-k}\delta_{h_k}(\lambda).
\end{eqnarray}
This is the estimate (\ref{Multi_Correction_Err_fun}). From Theorem
\ref{Rayleigh_Quotient_error_theorem} and
(\ref{Error_u_h_n_Multi_Correction}), we can obtain the estimate
(\ref{Multi_Correction_Err_eigen}).
\end{proof}

\section{Postprocessing algorithm}
Theorem \ref{Ts.8} has  only  theoretical value and cannot be used
in practice since the exact solution of the Stokes source problem
(\ref{auxstoke}) is  always not known. In order to make it useful,
we need to get a sufficient accurate approximation of the Stokes
source problem.
 Here we  discuss two possible ways how to obtain the approximation of the
 Stokes source problem (\ref{auxstoke}). The first way is the so-called
``two-grid method" of Xu and Zhou introduced and studied in
\cite{XuZhou} for second order differential equations and integral
equations. The second way proposed and studied by Andreev and
Racheva in \cite{RachevaAndreev} uses the same mesh but higher order
 finite element space.

The first way uses a finer mesh (with mesh size $h^2$) to get an
approximation of $ \tilde{\lambda}$ with an error $O(h^{4k})$ or
$O(h^{4s})$ for $k\leq 2$. The advantage of this approach is that it
uses the same finite element spaces and does not require higher
regularity of the exact eigenfunctions. The second way is based on
the same finite element mesh $\mathcal{T}_h$ but using one order
higher  finite element space. Also, to get an improvement for the
approximation of $\tilde{\lambda}_h$ to the  error $O(h^{4})$  or
$O(h^{2+2\gamma})$ from $O(h^2)$, we need to investigate the
regularity of the Stokes eigenvalue  problem.

We can  treat both ways in the same abstract manner. Namely, let us
introduce the enriched finite element space $\tilde{\bf
V}_h\times\tilde{Q}_h$ such that ${\bf V}_h\times Q_h\subset
\tilde{\bf V}_h\times \tilde{Q}_h\subset(H_0^1(\Omega))^2\times
L_0^2(\Omega)$ and consider the following discrete Stokes problem:

Find $(\tilde{\mathbf u}_h, \tilde{p}_h)\in \tilde{\bf V}_h\times
\tilde{Q}_h $ such that
\begin{equation}\label{dauxstokes}
\left\{\begin{array}{rcl}
a(\tilde{\mathbf u}_h, {\mathbf v}_h)-b({\bf
v}_h, \tilde{p}_h)&=&\lambda_hs({\mathbf u}_h,{\mathbf v}_h)\ \ \ \
\forall {\bf
v}_h\in \tilde{\bf V}_h,\\
b(\tilde{\mathbf u}_h, q_h)+S_h(\tilde{p}_h,q_h)&=&0\ \quad \quad
\quad \quad \ \ \ \ \forall q_h\in \tilde{Q}_h.
\end{array}
\right.
\end{equation}
Here, we suppose that the approximation $(\tilde{\mathbf u}_h,
\tilde{p}_h)\in \tilde{\bf V}_h\times \tilde{Q}_h$ has the following
error estimate:

For a smooth domain
\begin{align}\label{highapprox}
\|\tilde{\mathbf u}-\tilde{\bf
u}_h\|_1+\|\tilde{p}-\tilde{p}_h\|_0 \leq Ch^{k+1}(\|\tilde{\bf
u}\|_{k+2}+\|\tilde{p}\|_{k+1}),
\end{align}
and  for a convex polygonal domain
\begin{align}\label{highapproxs}
\|\tilde{\mathbf u}-\tilde{\bf u}_h\|_1+\|\tilde{p}-\tilde{p}_h\|_0
\leq Ch^{2s}(\|\tilde{\bf u}\|_{s+2}+\|\tilde{p}\|_{s+1}).
\end{align}

 So, we need define the following Rayleigh quotient for $(\tilde{\mathbf u}_h,\tilde{p}_h)$
\begin{align}\label{drayleighaux}
\tilde{\lambda}_h = \frac{a(\tilde{\mathbf u}_h, \tilde{\bf
u}_h)}{r(\tilde{\mathbf u}_h,\tilde{\mathbf u}_h)}.
\end{align}

From the analysis above, we can obtain the following error estimate
for the new eigenpair approximation $(\tilde{\mathbf u}_h, \tilde{p}_h, \tilde{\lambda}_h)\in
\tilde{\bf V}_h\times\tilde{Q}_h\times \mathbb{R}$.
\begin{theorem}\label{Ts.9}
Assume $\tilde{\lambda}_h$ is defined by (\ref{drayleighaux}),
$(\tilde{\mathbf u}_h, \tilde{p}_h)$ is the solution of
(\ref{dauxstokes}) and  $({\mathbf u}, p, \lambda)$ is the true eigenpair
 of the Stokes eigenvalue problem (\ref{wstokesproblem}).
Then we have
\begin{align}
&|\tilde{\lambda}_h-\lambda|\leq C(\|{\mathbf u}-{\bf
u}_h\|_{-1}+|\lambda-\lambda_h|+\|\tilde{\mathbf u}-\tilde{\mathbf
u}_h\|_1
+\|\tilde{p}-\tilde{p}_h\|_0)^2\nonumber\\
&\hskip4cm +CS_h(\tilde{p}_h,\tilde{p}_h),\label{errorofpost1}\\
&\|\tilde{\mathbf u}_h-{\mathbf u}\|_1+\|\tilde{p}_h-p\|_0\leq
C(\|{\mathbf u}-{\mathbf
u}_h\|_{-1}+|\lambda-\lambda_h|+\|\tilde{\bf u}-\tilde{\bf
u}_h\|_1\nonumber\\
&\hskip4cm+\|\tilde{p}-\tilde{p}_h\|_0).\label{errorofpost2}
\end{align}
\end{theorem}
\begin{proof}
First from (\ref{errorofpostsolution}) and the triangle inequality,
we can obtain (\ref{errorofpost2}). Using $b(\tilde{\bf
u}_h,\tilde{p}_h)+S_h(\tilde{p}_h,\tilde{p}_h)=0$ and (\ref{rayexpan}), the following error
estimate holds
\begin{eqnarray*}
|\tilde{\lambda}_h-\lambda|&\leq&C(\|\tilde{\mathbf u}_h-{\bf
u}\|_1^2+\|\tilde{p}_h-p\|_0^2)+S_h(\tilde{p}_h,\tilde{p}_h)\\
&\leq&C(\|{\mathbf u}-{\mathbf
u}_h\|_{-1}+|\lambda-\lambda_h|+\|\tilde{\mathbf u}-\tilde{\bf
u}_h\|_1+\|\tilde{p}-\tilde{p}_h\|_0)^2+S_h(\tilde{p}_h,\tilde{p}_h).
\end{eqnarray*}
This is the desired result (\ref{errorofpost1}) and we complete the proof.
\end{proof}

Now, we can present a practical postprocessing algorithm which can
improve the accuracy of  eigenpair approximations for  the Stokes
eigenvalue problem (\ref{wstokesproblem}).

{\it{\bf Algorithm 1.}

(1)\ Solve the discrete Stokes eigenvalue problem (\ref{dwstokesproblem}) for
$({\mathbf u}_h, p_h, \lambda_h)\in {\bf V}_h\times Q_h\times
\mathbb{R}$.

(2)\ Solve the discrete Stokes source  problem (\ref{dauxstokes})  to get the solution
$(\tilde{\mathbf u}_h, \tilde{p}_h)\in \tilde{\bf V}_h\times
\tilde{Q}_h$.

(3)\ Compute
$$\tilde{\lambda}_h=\frac{a(\tilde{\mathbf u}_h, \tilde{\mathbf u}_h)}{r(\tilde{\mathbf u}_h, \tilde{\mathbf u}_h)}.$$

The pair $(\tilde{\mathbf u}_h, \tilde{p}_h, \tilde{\lambda}_h)$
represent a new (and better than $({\mathbf u}_h, p_h, \lambda_h)$)
approximation to $({\mathbf u}, p, \lambda)$. }

Let us discuss two methods to construct the augmented finite element
space $\tilde{\bf V}_h\times \tilde{Q}_h$ for solving the Stokes
source problem (\ref{dauxstokes}).

{\it Way 1.} (``Two grid method" from \cite{XuZhou}):  In this case,
$\tilde{\bf V}_h\times \tilde{Q}_h$ is the same type of finite
element space as ${\bf V}_h\times Q_h$ on the finer mesh
$\tilde{\mathcal T}_h$ with mesh size $h^{\beta} (\beta>1)$. Here
$\tilde{\mathcal T}_h$ is a finer mesh of $\Omega$ which can be
generated by the refinement just as in the multigrid
method(\cite{XuZhou}).

First, let us  consider the case when the exact eigenfunction is
smooth and have the error estimate (\ref{2.23}) and (\ref{2.24}).
Because the maximum regularity of the solution $(\tilde{\mathbf u},
\tilde{p})$ of Stokes source problem (\ref{auxstoke}) is
$(H^3(\Omega))^2\times H^2(\Omega)$, we need to chose $k\leq 2$. In
this case, we obtain the following improved accuracy for the
eigenpair approximation when $\beta=2$(\cite{XuZhou})
\begin{align}
|\tilde{\lambda}_h-\lambda|&\leq Ch^{4k}\ \ \ \ \ {\rm for}\ k\leq 2,\\
 \|\tilde{\mathbf u}_h-{\bf
u}\|_1+\|\tilde{p}_h-p\|_0&\leq Ch^{2k}\ \ \ \ \ {\rm for}\ k\leq 2.
\end{align}
When $\Omega$ is a convex polygonal domain, with the error estimate
(\ref{2.20s}), (\ref{2.21s}) and Theorem \ref{Ts.9}, we have
\begin{align}
|\lambda-\tilde{\lambda}_h|&\leq Ch^{4s},\\
\|\tilde{\mathbf u}_h-{\mathbf
u}\|_1+\|\tilde{p}_h-p\|_0&\leq Ch^{2s},
\end{align}
where we also choose $\beta=2$. From the error estimate above, we can
find that the postprocessing method can obtain the convergence order as same
as solving the Stokes eigenvalue problem on the finer mesh
$\tilde{\mathcal T}_h$. This improvement costs solving the Stokes
source problem on a finer mesh with mesh size $O(h^2)$. This is
better than solving the Stokes eigenvalue problem on the finer mesh
directly, because solving Stokes source problem needs much less
computation than solving Stokes eigenvalue problem.

{\it Way 2.} (``Two space" method from \cite{RachevaAndreev}): In
this case, $\tilde{\bf V}_h\times \tilde{Q}_h$ is defined on the
same mesh $\mathcal{T}_h$ but one order higher than ${\bf V}_h\times
Q_h$. Since the maximum regularity of the solution $(\tilde{\mathbf
u}, \tilde{p})$ for the Stokes source problem (\ref{auxstoke}) is
$(H^3(\Omega))^2\times H^2(\Omega)$, we can only use  the first
order finite element space to solve the original Stokes eigenvalue
problem (\ref{dwstokesproblem}), and solve the Stokes source problem
(\ref{dauxstokes}) in the second order finite element space. So, we
only have the following error estimate for $({\mathbf u}_h, p_h,
\lambda_h)$
\begin{align}
|\lambda-\lambda_h|&\leq Ch^2,\label{errofk1}\\
\|{\mathbf u}-{\mathbf
u}_h\|_1+\|p-p_h\|_0&\leq Ch,\label{errofk1u}\\
\|{\mathbf u}-{\mathbf u}_h\|_{-1}&\leq Ch^2.\label{errofk1u-1_norm}
\end{align}

First, if the domain $\Omega$ is smooth, we have the following error
estimate
\begin{align}
|\lambda-\tilde{\lambda}_h|&\leq Ch^4,\\
\|{\mathbf u}-\tilde{\mathbf u}_h\|_1+\|p-\tilde{p}_h\|_0&\leq
Ch^2.
\end{align}
This is an obvious improvement than (\ref{errofk1}) and
(\ref{errofk1u}).

When $\Omega$ is a convex polygonal domain, from the regularity of
the Stokes source problem and the error estimates (\ref{highapproxs}), (\ref{errorofpost1}) and (\ref{errorofpost2}),
 we have
\begin{align}
|\lambda-\tilde{\lambda}_h|&\leq Ch^{2+2\gamma},\\
\|{\mathbf u}-\tilde{\mathbf u}_h\|_1+\|p-\tilde{p}_h\|_0&\leq
Ch^{1+\gamma}.
\end{align}
This estimate is also an obvious improvement than (\ref{errofk1})
and (\ref{errofk1u}).

The improved error estimate above just cost solving the Stokes
source problem on the same mesh in the second order finite element
space.

\section{Numerical results}
In this section, we give a  numerical example to illustrate the
efficiency of the postprocessing algorithm derived in this paper.
Since we do not know the exact solution of the Stokes eigenvalue
problems, the numerical results only give the behaviors of
eigenvalue approximations by the postprocessing algorithms.

We consider the Stokes eigenvalue problem (\ref{stokesproblem}) on
the domain $\Omega=(0,1)\times(0,1)$. From \cite{Wieners} and
\cite{ChenLin}, we choose  a sufficiently accurate first eigenvalue
approximation $\lambda=52.3446911$ as the first true one.

We first give numerical results of the postprocessing algorithm
which the enriched spaces constructed by refining the current mesh
by the regular way. Here we use the element
$(V_h,D_h)=(P_1,P_{-1}^{\rm disc})$  with
\begin{align*}
P_1 = &\big\{v\in H^1(\Omega): v|_K\in P_1(K),\ \forall
K\in\mathcal{T}_h\big\},\\
P_{-1}^{\rm disc}=&\big\{v\in L^2(\Omega): v|_K\in P_{-1}(K),\
\forall K\in\mathcal{T}_h\big\},
\end{align*}
to solve the Stokes eigenvalue problem (\ref{dwstokesproblem}) and
the Stokes source problem (\ref{dauxstokes}). The numerical results
are shown in Figure \ref{figure_finer_mesh}.
\begin{figure}[htbp]
\centering
\includegraphics[width=8cm,height=7cm]{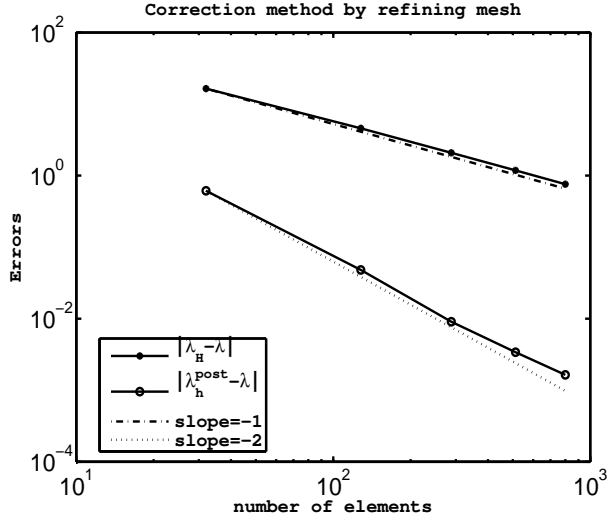}
\caption{Errors for refining mesh method with $\alpha_K=0.1$}\label{figure_finer_mesh}
\end{figure}
Then we give numerical results of the postprocessing algorithm which
the enriched spaces constructed by one order higher finite element.
We first solve the Stokes eigenvalue problem (\ref{dwstokesproblem})
by the lowest order stabilization element
$(V_h,D_h)=(P_1,P_{-1}^{\rm disc})$ and solve the Stokes source
problem (\ref{dauxstokes}) by second order stabilization element
$(V_h,D_h)=(P_2^+,P_{1}^{\rm disc})$
(\cite{MatthiesSkrzypaczTobiska_Oseen}) with
\begin{align*}
P_2^+ =&\big\{v\in H^1(\Omega) : v|_K\in P_2(K)\oplus \varphi_K
\cdot
P_1(K),\ \forall K\in\mathcal{T}_h\big\},\\
P_{1}^{\rm disc}=&\big\{v\in L^2(\Omega) : v|_K\in P_1(K),\ \forall
K\in\mathcal{T}_h\big\},
\end{align*}
 on the same triangular
meshes, where the bubble function $\varphi_K$ is defined by the
barycenter coordinates $\lambda_{1,K}, \lambda_{2,K}$ and
$\lambda_{3,K}$ on the element $K$ with
$\varphi_K:=\lambda_{1,K}\lambda_{2,K}\lambda_{3,K}$.

The numerical results are shown in Figure \ref{figure_higher_order}.
\begin{figure}[htbp]
\centering
\includegraphics[width=8cm,height=7cm]{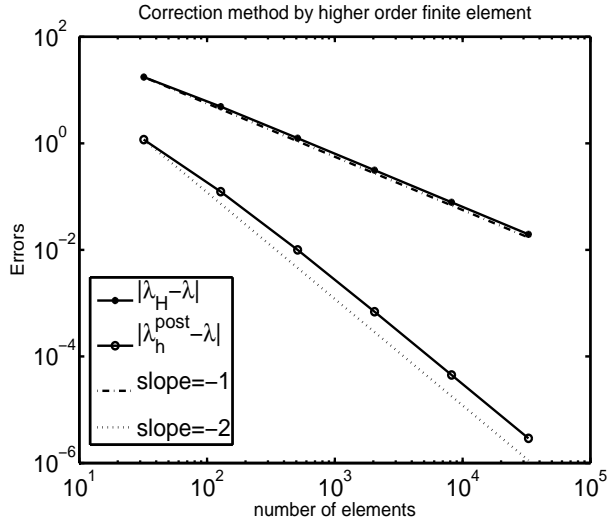}
\caption{Errors for higher order method with
$\alpha_K=0.1$}\label{figure_higher_order}
\end{figure}
From Figures \ref{figure_finer_mesh} and \ref{figure_higher_order}, we can find that the
postprocessing algorithm can improve the accuracy of the eigenvalue
approximations and confirm the theoretical analysis.

\section{Concluding remarks}
In this paper, the LPS method is applied to obtain the
approximations of Stokes eigenvalue problem and a type of
postprocessing method is also proposed to improve the convergence
order for the eigenpair approximation. The theoretical analysis is
given and the corresponding numerical examples are also used to
confirm the analysis. The postprocessing method proposed here can be
coupled with the adaptive mesh refinement in the two-grid method.
The application of LPS method makes the implementation of adaptive
mesh refinement more easily for solving Stokes eigenvalue problems
especially on the meshes with hanging nodes (\cite{Schieweck}).

In the future, we will extend our postprocessing method to the nonsymmetric Stokes eigenvalue problems which is more
general in the study of linearized stability for the Navier-Stokes equations (\cite{HeuvelineRannacher}).


\end{document}